\documentclass[12pt]{article}
\usepackage{amssymb}
\usepackage{array}
\usepackage{latexsym}
\usepackage{enumerate}
\usepackage{amsmath}
\usepackage{amsfonts}
\usepackage{amsthm}
\usepackage{graphicx}
\usepackage{comment}
\usepackage[english]{babel}
\usepackage{geometry}
\usepackage{hyperref}
\ifx\smallsetminus\undefined
\def\smallsetminus{\setminus}
\fi
\title{Intersection sets, three-character multisets and associated codes}
\author{A.~Aguglia\footnote{Dipartimento di Meccanica, Matematica e Management, Politecnico di Bari,
  Via Orabona 4, I-70126 Bari}, L.~Giuzzi\footnote{D.I.C.A.T.A.M., Universit\`a
  di Brescia, Via Branze 43, I-25123 Brescia}}
\date{}
\theoremstyle{plain}
\newtheorem{prop}{Proposition}[section]
\newtheorem{theorem}[prop]{Theorem}

\newtheorem{lemma}[prop]{Lemma}
\theoremstyle{definition}
\newtheorem{remark}[prop]{Remark}
\newcommand{\cH}{\mathcal H}
\newcommand{\cC}{\mathcal C}
\newcommand{\cQ}{\mathcal Q}

\newcommand{\cB}{\mathcal B}

\newcommand{\cF}{\mathcal F}

\newcommand{\PG}{\mathrm{PG}}
\newcommand{\AG}{\mathrm{AG}}
\newcommand{\GF}{\mathrm{GF}}
\newcommand{\cV}{\mathcal V}
\newcommand{\rank}{\operatorname{rank}}
\newcommand{\Rad}{\operatorname{Rad}}

\newcommand{\Tr}{\mathrm{Tr}\,}

\begin{document}
\maketitle
\begin{abstract}
In this article  we construct new  minimal intersection sets  in $\AG(r,q^2)$  sporting three intersection numbers with hyperplanes;
we then use these sets to obtain linear error correcting codes with
few weights, whose weight enumerator we  also determine.
Furthermore, we provide  a new
family of three-character multisets in $\PG(r,q^2)$ with $r$ even and we also
compute their weight distribution.
\end{abstract}
\noindent
{\bf Keywords}: Quadric, Hermitian variety, three-character set, multiset, error correcting code, weight enumerator
\par\noindent
{\bf MSC:} 51E20, 94B05

\section{Introduction}

Throughout this paper $q$ is taken to be any prime power.
A set of points $S$ in the  projective space $\PG(r,q^2)$ or in the affine space $\AG(r,q^2)$
 is a {\em $t$-intersection set} or  a {\em
$t$-fold blocking set with respect to hyperplanes} if every hyperplane
contains at least $t>0$ points of $S$.
A point $P$ of  a $t$-intersection set $S$ is said to be
{\em essential} if $S \setminus \{P\}$ is not a $t$-intersection set.
When all points of $S$ are essential then $S$ is {\em minimal}.

An intersection set $S$ in $\PG(r,q^2)$ or in $\AG(r,q^2)$ is an {\em $m$-character set } if the size of the intersection of $S$ with any hyperplane
might assume just one out of $m$ possible different values called {\em the characters} of $S$.

Sets with few characters
are connected with many theoretical and applied areas such as
coding theory, strongly regular graphs, association schemes, optimal multiple coverings, secret sharing; see in particular~\cite{BDG,CK,CP,
CaG,CaK,Cha,Di1,Di} for
applications of $2$-- and $3$--character sets.
For an extensive survey of results on three-character sets, see
also~\cite{Di0} and the references therein.

A {\em multiset} in  $\PG(r,q^2)$ is a mapping
$M : \PG(r,q^2) \rightarrow {\mathbb N}$ from  the points of $\PG(r,q^2)$ into non-negative
integers.
The \emph{points} of a multiset are the points $P$ of $\PG(r,q^2)$
with multiplicity $M(P) > 0$.
Certain multisets arise in various classification problems for optimal linear codes of higher dimension; see~\cite{Di,HK}.

We recall how a linear code in $q^2$ symbols is generated from a (multi)-set $\cV$ of points in $\PG(r,q^2)$.
Fix a reference frame in $\PG(r,q^2)$ and construct a matrix
$G$ by taking as columns the coordinates of the points of $\cV$
suitably normalized. The code $\cC$ having $G$ as generator matrix is
called the \emph{code generated from} $\cV$.

In the case in which $\cV$ is a set of points, that is $G$ does not
contain columns which are scalar multiples of each other, then
$\cC$ is the \emph{projective code generated from } $\cV$.
The spectrum of the intersections of $\cV$ with the
hyperplanes of $\PG(r,q^2)$ provides the list of the weights of the associated
code; we refer to~\cite{TVN} for further details on this
 geometric approach to codes.

As the order of the points in $\cV$ or their normalization
change, it is potentially possible to construct different codes from
the same set of points.  However, all of these are monomially equivalent;
thus, in the following discussion we shall speak of
\emph{the} code associated to a multiset; see~\cite{DoSi}.

The present paper is organized as follows. In Section~\ref{PRE} we recall a non-standard model of $\PG(r,q^2)$ which will be useful for our constructions.
In Section~\ref{s3} we consider  certain affine sets of $\AG(r,q^2)$
which allows to construct
interesting geometric objects with three characters.
Using these sets  in Section ~\ref{s4}  we then construct linear error correcting codes with
four weights and we determine the corresponding weight enumerator.
In Section~\ref{tcha} we present a construction of a $3$-character multiset in $\PG(r,q^2)$, for any $r$ even  and  we determine the weight distribution of the corresponding  three-weight code.

The study of the weights is important, since they measure the efficiency of the code and their knowledge is useful for decoding.

The codes we shall obtain in the present paper are all $q$-{\em divisible} that is they are $q$-ary code whose all non-zero weights are divisible by $q$;
see~\cite{DCS}.

\section{Preliminaries}
\label{PRE}
It is well known that all non--degenerate Hermitian varieties
of $\PG(r,q^2)$ are projectively
equivalent and that they sport just two intersection numbers with
hyperplanes; see~\cite{S}. Thus, non--degenerate Hermitian varieties
are  two-character set.
Quasi--Hermitian varieties $\cV$ of
$\PG(r,q^2)$ are combinatorial objects which have the same size and the
same intersection numbers with hyperplanes as  (non--degenerate) Hermitian
varieties $\cH$.

In~\cite{AA,ACK} new infinite families of quasi--Hermitian
varieties have been constructed by modifying some point-hyperplane
incidences in $\PG(r,q^2)$.
To this purpose, the authors kept the point-set of $\PG(r,q^2)$
but altered the geometry by suitably replacing the subspaces of
higher type.

The  following non-standard model $\Pi$  of $\PG(r,q^2)$, originally introduced in \cite{ACK},
leads to an extension to higher dimensional spaces of Buekenhout-Metz unitals
 and it shall also
be relevant for the current work.

Fix a non-zero element $a\in\GF(q^2)$.
For any choice $\mathbf{m}=(m_1,\ldots,m_{r-1})\in\GF(q^2)^{r-1}$ and
$d\in\GF(q^2)$  let $\cQ_a(\mathbf{m},d)$ denote the quadric of affine equation
\begin{equation}\label{quad}
x_r=a(x_1^2+\ldots+x_{r-1}^2)+m_1x_1+\ldots+m_{r-1}x_{r-1}+d.
 \end{equation}
 Consider now the birational transform $\AG(r,q^2)\to\AG(r,q^2)$ given by
 \[ \varphi_a: (x_1,\ldots,x_{r-1},x_r)\mapsto (x_1,\ldots,x_{r-1},x_r-a(x_1^2+\ldots+x_{r-1}^2)). \]
We can define a new geometry
$\Pi_a$ whose $t$-dimensional subspaces are the image  under $\varphi_a$ of the subspaces of $\AG(r,q^2)$  of dimension $t$ for
$0\leq t\leq r-1$.  As $\varphi_a$ is bijective, $\Pi_a$ is isomorphic
to $\AG(r,q^2)$. In particular, the set of the hyperplanes of $\Pi_a$ corresponds
to the set of all hyperplanes of $\AG(r,q^2)$ through $P_{\infty}(0,0,\ldots,0,1)$ together
with all of the quadrics  $\cQ_a({\mathbf m},d)$.
Completing $\Pi_a$ with its points at infinity in the usual way
we obtain a projective space isomorphic to $\PG(r,q^2 )$.

In \cite{AA}, an extension of Buekenhout-Tits unitals is considered,
 leading to non-isomorphic families of quasi-Hermitian varieties for $q$
 an odd power of $2$. However, we shall not be concerned any further
 with this second construction in the present paper.

\section{$3$-character sets in $\AG(r,q^2)$ }
\label{s3}
In this section we construct an infinite family of minimal intersection sets in $\AG(r,q^2)$ that sport just three intersection numbers.
Fix  a projective frame in $\PG(r,q^2)$ and assume the space to have
homogeneous coordinates $(X_0,X_1,\ldots, X_r)$.
Let $\AG(r,q^2)$ be the affine space obtained by taking as
hyperplane at infinity $\Pi_{\infty}$ of $\PG(r,q^2)$ that of
equation $X_0=0$. Then, the points of $\AG(r,q^2)$ have affine coordinates
$(x_1,x_2,\ldots,x_r)$ where $x_i=X_i/X_0$ for $i\in \{1,\ldots,r\}$.

Consider now the non--degenerate
Hermitian variety $\cH$ with affine equation
\begin{equation}
\label{eqherm}
x_r^q-x_r=(b^q-b)(x_1^{q+1}+\ldots+x_{r-1}^{q+1}),
\end{equation}
where $b\in\GF(q^2)\setminus \GF(q)$. The set of the points at
infinity of $\cH$ is
\begin{equation}
\label{heinf}
\cF=\{(0,x_1,\ldots,x_r)|x_1^{q+1}+\ldots+x_{r-1}^{q+1}=0\};
\end{equation}
that is $\cF$ is a Hermitian cone of $\PG(r-1,q^2)$,
projecting a Hermitian variety of $\PG(r-2,q^2)$ from
the point $P_{\infty}(0,\ldots,0,1)$.
In particular, the
hyperplane $\Pi_{\infty}$ is tangent to $\cH$ at $P_{\infty}$.

For any $a\in\GF(q^2)^*$ and $b\in\GF(q^2)\setminus\GF(q)$,
let $\cB(a,b)$ be the affine algebraic set   of  equation
\begin{equation}\label{eqqh}
x_r^q-x_r+a^q(x_1^{2q}+\ldots+x_{r-1}^{2q})-a(x_1^2+\ldots+x_{r-1}^2)=
(b^q-b)(x_1^{q+1}+\ldots+x_{r-1}^{q+1}).
\end{equation}
It is shown in~\cite{ACK} that $\cB(a,b)$,
together with the points at infinity of $\cH$, as given by
\eqref{heinf}, is a quasi--Hermitian variety $\cV$ of
 $\PG(r,q^2)$ provided that the following
algebraic conditions are satisfied:
for $q$ odd,
$r$ is odd and $4a^{q+1}+(b^q-b)^2 \neq 0$, or
$r$ is even and  $4a^{q+1}+(b^q-b)^2$ is a non--square in $\GF(q)$;
for $q$ even,
   $r$ is odd, or
  $r$ is even and $ \Tr_{q}(a^{q+1}/(b^q+b)^2)=0$.
Here
$\Tr_{q}$ with $q=2^h$,   denotes the absolute trace $\GF(q)\to \GF(2)$ which maps $x\in \GF(q)$
to $x+x^2+x^{2^2}+\ldots+x^{2^{h-1}}$.

We recall that for $r=2$ the condition that $4a^{q+1}+(b^q-b)^2$ is a non-square
in $\GF(q)$ for $q$ odd or $b\not\in\GF(q)$ and $\Tr_{q}(a^{q+1}/(b^q+b)^2)=0$ for $q$ even is known
 as Ebert's discriminant condition see~\cite{AGK,E}.

We shall study the point-set  $\cB(a,b)$
when complementary of conditions of those mentioned above hold.

\begin{table}
\[ \begin{array}{c|c|c|c}
     r & q & 4a^{q+1}+(b^q-b)^2  &\Tr_{q}(a^{q+1}/(b^q+b)^2) \\ \hline
     \text{odd} & \text{odd} & \text{non-zero} & \\
     \text{even} & \text{odd} &  \text{non-square in $\GF(q)$}& \\
     \text{odd} & \text{even} & & \text{Any} \\
     \text{even} & \text{even}& &0 \\
     \end{array} \]
 \caption{Summary of the cases considered in \cite[Theorem 3.1]{ACK} }
\label{tab1}
\end{table}
\begin{table}
\[ \begin{array}{c|c|c|c}
     r & q &4a^{q+1}+(b^q-b)^2 & Tr(a^{q+1}/(b^q+b)^2)   \\ \hline
     \text{odd}& \text{odd} & 0 &   \\
     \text{even}&\text{odd} & 0 &  \\
     \text{even} & \text{odd}&\text{non-zero square in $\GF(q)$} &  \\
     \text{even} & \text{even}& & 1\\
   \end{array} \]
\caption{Summary of the cases considered in Theorem \ref{mb3} and   Theorem \ref{mb1} }
\label{tab2}
\end{table}

\begin{table}[h]
\[ \begin{array}{c|c|c}
     r & q & \text{Case} \\ \hline
     r\equiv1\pmod{4} & q\equiv1\pmod4 & \ref{TA}) \\
     r\equiv1\pmod{4} & q\equiv3\pmod4 & \ref{TA}) \\
     r\equiv3\pmod{4} & q\equiv1\pmod4 & \ref{TA}) \\
     r\equiv3\pmod{4} & q\equiv3\pmod4 & \ref{TB}) \\
     r\equiv0\pmod{2} & q\equiv1\pmod2 & \ref{TC})
   \end{array}\]
 \caption{Cases for Theorem~\ref{t0}; $q$ odd.}
\label{tab3}
\end{table}

We are going to prove the following results.
\begin{theorem}\label{mb3}
\label{t0}
Suppose $q$ to be an odd prime-power and $4a^{q+1}+(b^q-b)^2=0$. Then,
$\cB(a,b)$ is a set of $q^{2r-1}$ points of
$\AG(r,q^2)$ with characters:
\begin{enumerate}
\item\label{TA}
 for $r\equiv 1\pmod 4$ or $r$ odd and $q\equiv 1\pmod 4$
\[ q^{2r-3}-q^{(3r-5)/2}, q^{2r-3}, q^{2r-3}-q^{(3r-5)/2}+q^{3(r-1)/2}; \]
\item\label{TB} for $r\equiv 3\pmod 4$ and $q\equiv 3\pmod 4$
\[q^{2r-3}+q^{(3r-5)/2}-q^{3(r-1)/2}, q^{2r-3}, q^{2r-3}+q^{(3r-5)/2}; \]
\item\label{TC} for $r$ even,
  \[q^{2r-3}-q^{(3r-4)/2}, q^{2r-3}, q^{2r-3}+q^{(3r-4)/2}.\]
 \end{enumerate}
 Furthermore, $\cB(a,b)$ is a minimal  intersection
 set with respect to hyperplanes for $r>2$.
\end{theorem}

\begin{theorem}\label{mb1}
  Let $r$ be even.
  Suppose that either $q$ is
  odd with $4a^{q+1}+(b^q-b)^2$  a non--zero square in $\GF(q)$ or
  $q$ is even and $\Tr_{q}(a^{q+1}/(b^q+b)^2)=1$.
  Then,  $\cB(a,b)$ is a set of $q^{2r-1}$ points of
  $\AG(r,q^2)$ with characters
  \[q^{2r-3}-q^{r-2}, q^{2r-3}, q^{2r-3}-q^{r-2}+q^{r-1}.\]
  $\cB(a,b)$  is also a minimal intersection set with respect to hyperplanes.
\end{theorem}
As it can be seen from Tables \ref{tab1} and \ref{tab2}, all of the
possibilities have been accounted for.
For the convenience of the reader, we also summarize the subcases of Theorem~\ref{t0} in Table~\ref{tab3}.

\subsection{Proof of Theorem~\ref{mb3}}
\label{PP3}
Recall that for any quadric $\cQ$, the radical $\Rad(\cQ)$ of $\cQ$
is the subspace
\[ \Rad(\cQ):=\{ x\in\cQ : \forall y\in\cQ, 
\langle x,y\rangle\subseteq\cQ \}, \]
where by $\langle x,y\rangle$ we denote the line through
$x$ and $y$.
It is well known that $\Rad(\cQ)$ is a subspace of $\PG(r,q^2)$.

Assume $\cB(a,b)$ to have Equation \eqref{eqqh}.
It is straightforward to see that $\cB(a,b)$ coincides with the
affine part of the Hermitian variety $\cH$  of equation \eqref{eqherm} in
the space $\Pi_{a}$;
hence,
any hyperplane $\pi_{P_{\infty}}$ of $\PG(r,q^2)$
passing through $P_{\infty}$ meets $\cB(a,b)$ in
$|\cH \cap \pi_{P_{\infty}}|=q^{2r-3}$ points.

Now we are interested in  the possible intersection sizes of $\cB(a,b)$ with
a generic hyperplane
\[ \pi: x_r=m_1x_1+\cdots+m_{r-1}x_{r-1}+d, \]
of $\AG(r,q^2)$ with coefficients $m_1,\ldots,m_r,d\in\GF(q^2)$.
This is the same as to study the intersection of $\cH$ with the
quadrics $\cQ_a(\mathbf{m},d)$ and hence we have to solve the system
  \begin{equation}
    \left\{\begin{array}{l}\label{sist1}
      x_r^q-x_r=(b^q-b)(x_1^{q+1}+\ldots+x_{r-1}^{q+1})   \\
       a(x_1^2+\ldots+x_{r-1}^2)+m_1x_1+\ldots+m_{r-1}x_{r-1}-x_r+d=0
      \end{array}\right.
  \end{equation}
  Recovering $x_r$ from the second equation in \eqref{sist1}
  and replacing its value in the first equation, we obtain the following
\begin{equation}
\begin{array}{l}\label{eqc}
a^q(x_1^{2q}+\ldots+x_{r-1}^{2q})+(b-b^q)(x_1^{q+1}+\ldots+x_{r-1}^{q+1})+m_1^qx_1^{q}+\\
+\ldots+m_{r-1}^qx_{r-1}^{q}+d^q-a(x_1^2+\ldots+x_{r-1}^2)-m_1x_1+\ldots+ \\
-m_{r-1}x_{r-1}-d=0 .
\end{array}
\end{equation}

As $q$ is odd, there is $\varepsilon\in\GF(q^2)\setminus\GF(q)$ such that
$\varepsilon^q=-\varepsilon$.
For any $x\in\GF(q^2)$ write $x=x^0+\varepsilon x^1$ with $x^0,x^1\in\GF(q)$; in this way
the previous Equation~\eqref{eqc} becomes
\begin{multline}
\label{eqodd1}
\displaystyle\sum_{i=1}^{r-1}\left((b^1+a^1)\varepsilon^2(x_i^1)^2+2a^0x_i^0x_i^1+(a^1-b^1)(x_1^0)^2\right)+\sum_{i=1}^{r-1}(m_i^0x_i^1+m_i^1x_i^0)+d^1=0;
\end{multline}
that is the number  $N$ of affine points which lie in
$\cB(a,b) \cap \pi$ is the same as the number of points of the
affine quadric $\cQ$ of $\AG(2r-2,q)$ of Equation~\eqref{eqodd1}.
Following the approach of~\cite{ACK}, in order to compute $N$, we
first count  the number of points of the quadric
at infinity $\cQ_{\infty}:=\cQ\cap\Pi_{\infty}$
of $\cQ$ and then we determine $N=|\cQ|-|\cQ_{\infty}|$.
Observe that the quadric $\cQ_{\infty}$ of $\PG(2r-3,q)$ has
a block matrix of the form
\begin{equation}\label{matr}
A_{\infty}=\begin{pmatrix}
    (a^1-b^1) &a^0 \\
a^0 & (b^1+a^1)\varepsilon^2 &   & & \\
 && \ddots &&  \\
     &  &  & (a^1-b^1) &a^0\\
     &&  &  a^0& (b^1+a^1)\varepsilon^2\\
      \end{pmatrix}
  \end{equation}
   with determinant
   $$\det A_{\infty}=(\varepsilon^2[(a^1)^2-(b^1)^2]-(a^0)^2)^{r-1}.$$
   Since $ (a^0)^2-\varepsilon^2[(a^1)^2-(b^1)^2]=a^{q+1}+(b^q-b)^2/4$ and
     $4a^{q+1}+(b^q-b)^2=0$ we have $\det A_{\infty}=0$. This means
      \[ \det\begin{pmatrix}
        (a^1-b^1) & a^0 \\
        a^0 & (a^1+b^1)\varepsilon^2
      \end{pmatrix}=0, \]
      that is, each of the $2\times 2$ blocks on the main
      diagonal of $A_{\infty}$ has rank $1$.
      Consequently, the rank of the matrix $A_{\infty}$ is exactly $r-1$.

      If $a^1=b^1$, then $a^0=0$, the matrix $A_{\infty}$ is
      diagonal and the quadric $\cQ_{\infty}$ is projectively equivalent to
      \[ (x_1^1)^2+(x_2^1)^2+\cdots+(x_{r-1}^1)^2=0. \]
      Otherwise, take
      \[ M=\begin{pmatrix}
        1 & 0 \\
        -a^0/(a^1-b^1) & 1 \\
        & & \ddots \\
        &&& 1 & 0 \\
        &&&  -a^0/(a^1-b^1) & 1
      \end{pmatrix}; \]
a direct computation proves that
\[ M^TA_{\infty}M=\begin{pmatrix}
  a^1-b^1 & 0 \\
   0 & 0 \\
   && \ddots \\
   &&& a^1-b^1 & 0 \\
   &&& 0 & 0
   \end{pmatrix}. \]
   Hence, $\cQ_{\infty}$ is projectively equivalent to the quadric of
   rank $r-1$ with equation
   \[ (x_1^0)^2+(x_2^0)^2+\cdots+(x_{r-1}^0)^2=0. \]
For $r$ odd,  we have that $\cQ_{\infty}$
is either
\begin{itemize}
\item a cone with vertex $\Rad(\cQ_{\infty})\simeq\PG(r-2,q)$
  and basis a hyperbolic
  quadric $Q^+(r-2,q)$ if $q\equiv 1\pmod 4$ or
  $r\equiv 1\pmod 4$, or
\item a cone with vertex  $\Rad(\cQ_{\infty})\simeq\PG(r-2,q)$
  and basis an elliptic
  quadric $Q^-(r-2,q)$ if $q\equiv 3\pmod 4$ and $r\equiv 3\pmod 4$; see~\cite[Theorem 1.2]{HT}.
\end{itemize}
For $r$ even,  $\cQ_{\infty}$ is a cone with vertex
$\Rad(\cQ_{\infty})\simeq\PG(r-2,q)$  and basis a parabolic
quadric $Q(r-2,q)$.

We now move to investigate the quadric $\cQ$. Clearly, the rank of
its matrix is either
$r-1$, $r$ or $r+1$.

  Write $\Pi_{\infty}=\Sigma\oplus\Rad(\cQ_{\infty})$. As $\Sigma$ is
  disjoint from the radical of the quadratic form inducing
  $\cQ_{\infty}$, we have that $\cQ'_{\infty}:=\Sigma\cap\cQ_{\infty}$ is a nondegenerate
  quadric (either hyperbolic, elliptic or parabolic according to
  the various conditions).

  When $\cQ$ has the same rank  $r-1$ as $\cQ_{\infty}$, we have
  \[\dim\Rad(\cQ)=\dim\Rad(\cQ_{\infty})+1.\] Observe that
  $\Rad(\cQ)\cap\Pi_{\infty}\leq\Rad(\cQ_{\infty})$.
  Thus, $\Rad(\cQ)\cap\Sigma=\{\mathbf{0}\}$ and
  $\Sigma$ is also a direct complement of $\Rad(\cQ)$.
  It follows that $\cQ$ is cone of vertex a
  $\PG(r-1,q)$ and
  basis a quadric of the same kind as $\cQ'_{\infty}$.
  If $\cQ$ has rank $r+1$, then the hyperplane at infinity
  is tangent to $\cQ$; in particular $\cQ$ must have as radical a
  $\PG(r-3,q)$; by \cite[Lemma 1.22]{HT},  the basis $\cQ''$ of $\cQ$ must have
  the same character (elliptic, parabolic  or hyperbolic) as $\cQ'_{\infty}$. 

  In the case in which the matrix of $\cQ$ has rank $r$, $\Rad(\cQ)=\Rad(\cQ_{\infty})$
  and
   $\cQ$ is a cone of vertex a $\PG(r-2,q)$ and basis a parabolic
  quadric $Q(r-1,q)$ for $r$ odd or $\cQ$ is a cone of vertex a $\PG(r-2,q)$
  and basis a hyperbolic  quadric $Q^+(r-1,q)$ or an elliptic
  quadric $Q^-(r-1,q)$ for $r$ even.
We can now write the complete list of sizes for $r$ odd:
    \[ |\cQ_{\infty}|=\frac{q^{2r-3}-1}{q-1}\pm q^{(3r-5)/2};\]
in case $\rank(\cQ)=r-1$, then
    \[ |\cQ|=\frac{q^{2r-2}-1}{q-1}\pm q^{3(r-1)/2}; \]
in case $\rank(\cQ)=r$,
    \[
    |\cQ|=\frac{q^{2r-2}-1}{q-1}; \]
in case $\rank(\cQ)=r+1$, then
    \[ |\cQ|=\frac{q^{2r-2}-1}{q-1}\pm q^{(3r-5)/2}. \]

    In particular, the possible values for $N=|\cQ|-|\cQ_{\infty}|$ are
\[ q^{2r-3}, q^{2r-3}+q^{3(r-1)/2}-q^{(3r-5)/2}, q^{2r-3}-q^{(3r-5)/2} \]
for $q\equiv1\pmod 4$ or $r\equiv 1\pmod 4$ and
\[ q^{2r-3}-q^{3(r-1)/2}+q^{(3r-5)/2}, q^{2r-3}+q^{(3r-5)/2} \]
for $q\equiv 3\pmod4$ and $r\equiv 3\pmod4$.\\
When $r$ is even we get:
    \[ |\cQ_{\infty}|=\frac{q^{2r-3}-1}{q-1};\]
in case $\rank(\cQ)=r-1$ or $\rank(\cQ)=r+1$, then
    \[ |\cQ|=\frac{q^{2r-2}-1}{q-1}; \]
    in case $\rank(\cQ)=r$,
    \[
    |\cQ|=\frac{q^{2r-2}-1}{q-1}\pm q^{(3r-4)/2}. \]
Thus, the possible list of cardinalities for $N=|\cQ|-|\cQ_{\infty}|$ is
\[ q^{2r-3}, q^{2r-3}+q^{(3r-4)/2}, q^{2r-3}-q^{(3r-4)/2}. \]

Now we are going to show that $\cB(a,b)$ is a minimal
  intersection set.  First of all,
  we prove that for any $P\in \cB(a,b)$ there exists a subspace
  $\Lambda_n(P)\simeq\AG(n,q^2)$, $1\leq n\leq r-1$ through $P$
  such that $|\cB(a,b) \cap \Lambda_n(P)|\leq q^{2n-1}-q^{n-1}$.
  The argument is by induction on $n$.
  Assume $n=1$. Then, for any $P \in \cB(a,b)$ there exists at
  least one line $\ell$ through $P$
   such that $|\ell\cap\cB(a,b)|<q$, otherwise $\cB(a,b)$ would contain more than
   $q^{2r-1}$ points.   Suppose now  that
   the result holds for $n=1,\ldots,r-2$, take $P\in \cB(a,b)$ and suppose
   that any hyperplane $\pi$ through $P$ meets $\cB(a,b)$ in at least
   $q^{2r-3}$ points. By induction,  there exists a subspace
   $\pi':=\Lambda_{r-2}(P)\simeq\AG(r-2,q^2)$ through $P$ meeting
   $\cB(a,b)$ in at most $q^{2r-5}-q^{r-3}$
   points. By considering all hyperplanes containing $\pi'$ we get
   $|\cB(a,b)|\geq (q^2+1)(q^{2r-3}-q^{2r-5}+q^{r-3})+ q^{2r-5}-q^{r-3}>
   q^{2r-1}$, a contradiction. Thus,
   through any $P\in\cB(a,b)$ there exists a hyperplane meeting $\cB(a,b)$ in
   $(q^{2r-3}-q^{(3r-5)/2})$ points for $r$ odd or  $(q^{2r-3}-q^{(3r-4)/2})$
   for $r$ even. This implies
   that  $\cB(a,b)$ is a  minimal intersection set for any $r>2$.

\begin{remark}
The quadric $\cQ_a({\mathbf{m}},d)$ of Equation \eqref{quad} shares its
tangent hyperplane at $P_{\infty}$ with the Hermitian variety \eqref{eqherm}.

The problem of the intersection
of the  Hermitian variety $\cH$  with  irreducible quadrics $\cQ$
having the same tangent plane at a common
point $P\in \cQ \cap \cH$ has been considered in detail for $r=3$
in~\cite{AG1,AG2}.
\end{remark}

\subsection{Proof of Theorem~\ref{mb1}}
\label{Ist}

First consider the case $q$ odd. Arguing as in the proof of Theorem \ref{mb3} we have that
any hyperplane $\pi_{P_{\infty}}$ of $\PG(r,q^2)$
passing through $P_{\infty}$ meets $\cB(a,b)$ in
$q^{2r-3}$ points.

In order to determine the possible intersection sizes of $\cB(a,b)$ with a hyperplane which does not pass through
$P_{\infty}$, say  $\pi: x_r=m_1x_1+\cdots+m_{r-1}x_{r-1}+d $, we need to compute the number $N$
 of affine points of the quadric $\cQ$ in $\AG(2r-2,q)$ with equation \eqref{eqodd1}.
We first discuss the nature of $\cQ_{\infty}=\cQ \cap \Pi_{\infty}$ whose associated matrix $A_{\infty}$ is of the form
\eqref{matr}.

Observe that, under our assumptions, for $q$ odd
$(-1)^{r-1}\det A_{\infty}$ is always a square in $\GF(q)$; hence,
$\cQ_{\infty}$ is a hyperbolic quadric of $\PG(2r-3)$.

 For $q$ even, choose $\varepsilon \in \GF(q^2) \setminus
 \GF(q)$ such that $\varepsilon^2+\varepsilon+\nu=0$, for some
 $\nu\in \GF(q)\setminus \{1\}$ with $\Tr_{q}(\nu)=1$.
  Then, $\varepsilon^{2q}+\varepsilon^q+\nu=0$. Therefore,
 $(\varepsilon^q+\varepsilon)^2+(\varepsilon^q+\varepsilon)=0$,
 whence $\varepsilon^q+\varepsilon+1=0$. With
 this choice of $\varepsilon$, the system given by~\eqref{eqqh}
 and~\eqref{quad}
 reads as
\begin{equation}
\begin{array}{l}
\label{eqeven1}
(a^1+b^1)(x_1^0)^2+[(a^0+a^1)+\nu(a^1+b^1)](x_1^1)^2+b^1x_1^0x_1^1+m_1^1x_1^0+(m_1^0+m_1^1)x_1^1\\
+\ldots
+(a^1+b^1)(x_{r-1}^0)^2+[(a^0+a^1)+\nu(a^1+b^1)](x_{r-1}^1)^2
+b^1x_{r-1}^0x_{r-1}^1\\
+m_{r-1}^1x_{r-1}^0
+(m_{r-1}^0+m_{r-1}^1)x_{r-1}^1+
d^1=0.
 \end{array}
 \end{equation}
The discussion of the (possibly degenerate) quadric $\cQ$ of Equation
\eqref{eqeven1} may be carried out in close analogy to what has been done
before.

Observe however that, as also pointed out in
the remark before \cite[Theorem 1.2]{HT},
some  caution is needed when quadrics
are studied and classified in even characteristic. Indeed, let
\[ A_{\infty}=\begin{pmatrix} 2(a^1+b^1) & b_1 \\
                                b_1   & 2((a^0+a^1)+\nu(a^1+b^1)) \\
                              & & \ddots \\
                            \end{pmatrix} \]
                          be the formal matrix
associated to the quadric $\cQ_{\infty}$ of equation
\[
(a^1+b^1)(x_1^0)^2+[(a^0+a^1)+\nu(a^1+b^1)](x_1^1)^2+b^1x_1^0x_1^1+\ldots\]
\[+(a^1+b^1)(x_{r-1}^0)^2+[(a^0+a^1)+\nu(a^1+b^1)](x_{r-1}^1)^2+b^1x_{r-1}^0x_{r-1}^1=0.\]
Its determinant is equal to
$$\det A_{\infty}=[4(a^1+b^1)(a^0+a^1+\nu(a^1+b^1))+(b^1)^2]^{r-1}.$$
In order to encompass the case $q$ even,
$\det A_{\infty} $ needs to be regarded
as a polynomial function in the ring $\mathbb{Z}[z_0,z_1,z_2,z_3]$ where  the terms
 $(a^0,a^1,b^0,b^1)$ are replaced by indeterminates $z_0,z_1,z_2,z_3$; then we regard it over
$\GF(q)$ for $(z_0,z_1,z_2,z_3)=(a^0,a^1,b^0,b^1)$.
This gives $\det A_{\infty} = b_1^{2(r-1)}$.
Here $b_1\neq 0$ since, by our assumption, $b^q\neq b$.
From~\cite[Theorem 1.2 (i)]{HT}, the quadric $\cQ_\infty$
must be non-degenerate.
Furthermore, by~\cite[Theorem 1.2 (ii)]{HT} and the successive Lemma
22.2.2 the nature of $\cQ_\infty$ can be ascertained as follows.
Let
$B$ the matrix
\begin{equation}\label{matr}
B=\begin{pmatrix}
    0 &b^1 \\
-b^1 & 0 &   & & \\
 && \ddots &&  \\
     &  &  & 0 &b_1\\
     &&  &  -b^1& 0\\
      \end{pmatrix}
  \end{equation}
and define
$$\alpha=\frac{\det B-(-1)^{r-1}\det A_{\infty}}
{4\,\det B}.$$
A straightforward computation shows that
$$\alpha=\frac{(b^1)^{2(r-1)}+ (4(a^1+b^1)(a^0+a^1+\nu(a^1+b^1))+(b^1)^2)^{r-1}}
{4\,(b^1)^{2(r-1)}}.$$
Here $\alpha$ has to be regarded 
as the quotient of two polynomials in $\mathbb{Z}[z_0,z_1,z_2,z_3]$  where  the terms
$(a^0,a^1,b^0,b^1)$ are replaced by indeterminates $z_0,z_1,z_2,z_3$ and, then, evaluated it over
$\GF(q)$ for $(z_0,z_1,z_2,z_3)=(a^0,a^1,b^0,b^1)$.
In particular, we get
$$\alpha=\frac{(a^1+b^1)(a^0+a^1+\nu(a^1+b^1))}{(b^1)^2}.$$
Arguing as in~\cite[p. 439]{ACK}, we see that
$\Tr_{q}(\alpha)=0$ and, hence, $\cQ_{\infty}$
is hyperbolic also for $q$ even.

We investigate
the possible nature of $\cQ$ in either case $q$ odd and $q$ even.
Suppose $\cQ$ to be non-singular; then $\cQ$ is a parabolic quadric and
\[N =|\cQ|-|\cQ_{\infty}|=\frac{(q^{r-1}+1)(q^{r-1}-1)}{q-1} -
  \frac{(q^{r-1}+1)(q^{r-2}-1)}{q-1}=q^{r-2}(q^{r-1}+1).\]
   If $\cQ$ is singular, then $\cQ$ is a cone  with vertex a point and basis a hyperbolic quadric; thus
   \[N= |\cQ|-|\cQ_{\infty}|=\frac{q(q^{r-1}+1)(q^{r-2}-1)}{q-1}-\frac{(q^{r-1}+1)(q^{r-2}-1)}{q-1}+1=\]
   \[=q^{r-2}(q^{r-1}+1)-q^{r-1}.\]
This gives the possible intersection numbers.

Finally, in order to show that $\cB(a,b)$ is a minimal
$(q^{2r-3}-q^{r-2})$--fold blocking set we can use the same techniques as
those adopted to prove that
$\cB(a,b)$ is a minimal blocking set in Theorem~\ref{mb3}.

\section{$4$-weight $q$-ary codes}
\label{s4}

Throughout this section $q$ is and odd prime power and $4a^{q+1}+(b^q-b)^2=0$ for any
$a\in \GF(q^2)^*$ and $b\in \GF(q^2)\setminus \GF(q)$. Let $\cB(a,b)$ the affine set of equation
\eqref{eqqh}. We prove the following theorem.
\begin{theorem}
  \label{mb3p}
  The points of $\cB(a,b)$ determine a projective code $\cC$ of
  length $n=q^{2r-1}$, dimension $k=r+1$ and weight enumerator
  $w(x):=\sum_i A_ix^i$ where
  \[ A_0=1,\qquad A_{q^{2r-1}}=(q^2-1) \]
  and all of the remaining $A_i$'s are $0$ with the exception of
  \begin{itemize}
  \item
   for $r\equiv 1\pmod 4$ or $r$ odd and $q\equiv 1\pmod 4$,
\[ A_{q^{2r-1}-q^{2r-3}-q^{3(r-1)/2}+q^{(3r-5)/2}}=(q^{r+1}-q^r)(q^2-1)\]
    \[ A_{q^{2r-1}-q^{2r-3}}=q^{2r}-q^2+(q^{2r}-q^{r+1})(q^2-1), \quad
     A_{q^{2r-1}-q^{2r-3}+q^{(3r-5)/2}}=q^{r+2}-q^r; \]

   \item for $r\equiv 3\pmod 4$ and $q\equiv 3\pmod 4$,
  \[ A_{q^{2r-1}-q^{2r-3}-q^{(3r-5)/2}}=(q^{r+1}-q^r)(q^2-1),\quad
     A_{q^{2r-1}-q^{2r-3}}=q^{2r}-q^2-q^{r+1}(q^2-1) \]
     \[ A_{q^{2r-1}-q^{2r-3}+q^{3(r-1)/2}-q^{(3r-5)/2}}=q^{r+2}-q^r; \]

   \item  for $r$ even,
\[ A_{q^{2r-1}-q^{2r-3}+q^{(3r-4)/2}}=A_{q^{2r-1}-q^{2r-3}-q^{(3r-4)/2}}=\frac{1}{2}(q^{r+1}-q^r)(q^2-1), \]
\[ A_{q^{2r-1}-q^{2r-3}}=q^{2r}+q^{r+2}-q^{r}-q^2+(q^{2r}-q^{r+1})(q^2-1). \]
\end{itemize}
In particular, each of these codes has exactly $4$ non-zero weights.
\end{theorem}

\begin{proof}
We begin by proving the following lemma.
\begin{lemma}\label{corol}
  The number of hyperplanes $N_j$ meeting $\cB(a,b)$ in exactly $j$
  points are as follows:
  \begin{enumerate}[{\rm (a)}]
  \item\label{CaA}
  For $r \equiv 1\pmod{4}$,
                   or $r$ odd and $q\equiv 1\pmod{4}$
   \[ N_{q^{2r-3}+q^{3(r-1)/2}-q^{(3r-5)/2}}=q^r,\qquad
      N_{q^{2r-3}}=\frac{q^{2r}-1}{q^2-1}-1+q^{2r}-q^{r+1},\]
      \[
       N_{q^{2r-3}-q^{(3r-5)/2}}=q^{r+1}-q^r.
         \]

\item \label{CaB}For $r\equiv 3\pmod4$ and $q\equiv 3\pmod4$
 \[ N_{q^{2r-3}+q^{(3r-5)/2}}=q^r,\qquad
      N_{q^{2r-3}}=\frac{q^{2r}-1}{q^2-1}-1+q^{2r}-q^{r+1},\]
      \[
       N_{q^{2r-3}-q^{3(r-1)/2}+q^{(3r-5)/2}}=q^{r+1}-q^r.
         \]
\item \label{CaC} For  $r$ even,
\[ N_{q^{2r-3}-q^{(3r-4)/2}}=\frac{1}{2}(q^{r+1}-q^r) \qquad
       N_{q^{2r-3}}=q^r+\frac{q^{2r}-1}{q^2-1}-1+q^{2r}-q^{r+1},\]
       \[
       N_{q^{2r-3}+q^{3(r-4)/2}}=\frac{1}{2}(q^{r+1}-q^r).
       \]
\end{enumerate}
\end{lemma}
\begin{proof} From the proof of Theorem \ref{mb1} it follows that in order to prove  Cases (\ref{CaA}) and (\ref{CaB}) we need to count the number of vectors $v:=(m_1^0,m_1^1,\ldots,m_{r-1}^0,m_{r-1}^1,d_1)\in\GF(q)^{2r-1}$ such that the matrix
 \[ A:=\begin{pmatrix}
   A_{\infty} & \begin{array}{c} m_1^0 \\ m_1^1 \\ \vdots \\ m_{r-1}^1 \end{array} \\
   \begin{array}{cccc} m_1^0 & m_1^1 & \cdots & m_{r-1}^1 \end{array} & d_1
   \end{pmatrix} \]
 of $\cQ$ with equation \eqref{eqodd1} has respectively rank $r-1$, $r$ or
$r+1$.

We  observe  that
$A$ has rank $r-1$ if, and only if, there exist a scalar $\lambda$
such that for all
$i=1,\ldots,r-1$ we have
$m_i^1=\lambda m_i^0$; also, the value of $d_1$ turns out to be uniquely
  determined.
  Thus, the number of distinct possibilities for the parameters $m_1,\ldots,m_{r-1},d$
  is exactly $q^r$. The rank of the  matrix of $\cQ$ is at least
  $r$ in the remaining $q^{2r}-q^{r}$ cases.
  Suppose it to be $r+1$. This means that the column $(m_1^0,m_1^1,\ldots,m_{r-1}^0,m_{r-1}^1)^T$
  is linearly independent from the columns of $A_{\infty}$; so, there are
  $q^{2r-2}-q^{r-1}$ ways to choose $m_1^0,\ldots,m_{r-1}^1$. Furthermore, for any such choice
  the vector $v=(m_1^0,\ldots,m_{r-1}^1,d_1)$ is also independent from the first $2r-2$ rows of
  $A$. So the overall number of planes with such property is $q^2(q^{2r-2}-q^{r-1})=q^{2r}-q^{r+1}$.
  The remaining $q^{r+1}-q^r$ choices yield a matrix of rank $r$.

  In Case (\ref{CaC}), again from the proof of Theorem \ref{mb1}
  when $r$ is even, we need to count how  often $\cQ$ with equation \eqref{eqodd1} turns out to be elliptic
  rather than hyperbolic.
  For any choice of the parameters $m_1,\ldots,m_{r-1},d$ there is
  exactly one quadric $\cQ$ to consider.
  As $\cQ_{\infty}$ is always a parabolic quadric, we can
  assume it to be fixed.
  Denote by $\sigma^0,\sigma^+,\sigma^-$ respectively the number of
  quadrics $\cQ$ which are parabolic, elliptic or hyperbolic. Clearly
  $\sigma_0$ corresponds to the case in which $\rank(\cQ)=\rank(\cQ_{\infty})$ or
  $\rank(\cQ)=\rank(\cQ_{\infty})+2$.
  We have
  \[ \sigma^++\sigma^0+\sigma^-=q^{2r},\qquad \sigma^0=q^{2r}-q^{r+1}+q^r. \]
  Each point of $\cB(a,b)$ lies on $\frac{q^{2r}-1}{q^2-1}$ hyperplanes;
  of these $\frac{q^{2r-2}-1}{q^2-1}$ pass through $P_{\infty}$ (and
  they must be discounted).
  Thus, we get
  \begin{multline*}
   q^{2r-2}|\cB|=q^{4r-3}=\sigma^0q^{2r-3}+\sigma^+(q^{2r-3}+q^{(3r-4)/2})+\sigma^-(q^{2r-3}-q^{(3r-4)/2})= \\
    q^{2r-3}(\sigma^0+\sigma^++\sigma^-)+q^{(3r-4)/2}(\sigma^+-\sigma^-)=
    q^{4r-3}+(\sigma^+-\sigma^-)q^{(3r-4)/2}.
  \end{multline*}
  Hence, $\sigma^+=\sigma^-=\frac{1}{2}(q^{r+1}-q^r)$.
\end{proof}
If we regard $\cB(a,b)$ as a set of points in $\PG(r,q^2)$, then we can consider
the projective code $\cC$  of length $q^{2r-1}$ and dimension $r+1$ generated from $\cB(a,b)$. Denote by $A_j$ the number
of codewords of $\cC$ of weight $j$. Observe that a hyperplane $\pi$
meeting $\cB(a,b)$ in $n$ points always determines $(q^2-1)$ codewords
of weight $q^{2r-1}-n$.
As
the hyperplane at infinity is disjoint from $\cB(a,b)$, we have
\[ A_{q^{2r-1}}=(q^2-1). \]
The remaining weights follow from Lemma~\ref{corol}.
This completes the proof of Theorem~\ref{mb3p}.

\end{proof}

\section{$3$-character multisets in $\PG(r,q^2)$, $r$ even and $3$-weight codes}

\label{tcha}
We keep all previous notation. In \cite[Theorem 4.1]{AK} it is shown that  for $r=2$, $q$ odd and
$4a^{q+1}+(b^q-b)^2\neq 0$ or $r=2$, $q$ even
and $\Tr_{q}(a^{q+1}/(b^q+b)^2)=1$, the set $\cB(a,b)$ can be completed
to a $2$--character  multiset $\overline{\cB}(a,b)$ yielding a two-weight code.

Here we prove that using a similar technique we can construct two
infinite families of three-weight codes. The construction is as follows.

 Let $r$ be even.
  Suppose that either $q$ is
  odd with $4a^{q+1}+(b^q-b)^2$  a non-zero square in $\GF(q)$ or
  $q$ is even and $\Tr_{q}(a^{q+1}/(b^q+b)^2)=1$.
 From Theorem \ref{mb1}  $\cB(a,b)$ is a set of $q^{2r-1}$ points of
  $\AG(r,q^2)$ with characters $q^{2r-3}-q^{r-2}$, $q^{2r-3}$, $q^{2r-3}-q^{r-2}+q^{r-1}$.

Now consider the multiset $\overline{\cB}(a,b)$ in $\PG(r,q^2)$ arising
from $\cB(a,b)$ by assigning multiplicity larger than $1$ to the point $P_{\infty}$.

More in detail  the points of the  $3$--character multiset
$\overline{\cB}(a,b)$  are exactly those of $\cB(a,b)\cup\{P_{\infty}\}$
where each affine point of ${\cB}(a,b)$  has multiplicity one, and
$P_{\infty}$   has multiplicity $j$.
In this way  $\overline{\cB}(a,b)$ turns out to have the following characters:
 \[j,q^{2r-3}+j, q^{2r-3}-q^{r-2}, q^{2r-3}-q^{r-2}+q^{r-1}.\]

The linear code $\cC$
 associated with $\overline{\cB}(a,b)$.
 is a
$[q^{2r-1}+j, r+1]_{q^2}$  code with weights
\[ q^{2r-1}, q^{2r-1}-q^{2r-3},q^{2r-1}-q^{2r-3}+q^{r-2}+j,q^{2r-1}-q^{2r-3}+q^{r-2}-q^{r-1}+j \]
For $j=q^{r-1}-q^{r-2}$ of $j=q^{2r-3}-q^{r-2}$ this is a $3$--weight code.

For $j=q^{r-1}-q^{r-2}$, the only hyperplane meeting $\overline{\cB}(a,b)$ in $j$ points is
that at infinity; thus $N_{q^{r-1}-q^{r-2}}=1$;
the hyperplanes  meeting  $\overline{\cB}(a,b)$ in $q^{2r-3}-q^{r-2}+q^{r-1}$ points are the hyperplanes passing through $P_{\infty}$ together with the hyperplanes for which the corresponding quadric $\cQ$ of equation \eqref{eqeven1} is singular.
Therefore,
 \[N_{q^{2r-3}-q^{r-2}+q^{r-1}}=\frac{q^{2r}-1}{q^2-1}+q^{2r-1}.\]
The remaining hyperplanes intersect $\overline{\cB}(a,b)$ in $q^{2r-3}-q^{r-2}$  points and hence
\[N_{q^{2r-3}-q^{r-2}}=q^{2r}-q^{2r-1}-1. \]
Thus  the  weight enumerator of $\cC$ is 
  $w(x):=\sum_i A_ix^i$ where
  \[ A_0=1,\qquad A_{q^{2r-1}}=q^2-1 \]
  and all of the remaining $A_i$'s are $0$ with the exception of
 \[ A_{q^{2r-1}-q^{2r-3}}=q^{2r}-1+q^{2r-1}(q^2-1), \quad
     A_{q^{2r-1}-q^{2r-3}+q^{r-1}}=(q^{2r}-q^{2r-1}-1)(q^2-1). \]
For $j=q^{2r-3}-q^{r-2}$ a similar argument gives
\[ N_{2q^{2r-3}-q^{r-2}}=\frac{q^{2r}-1}{q^2-1},\qquad
 N_{q^{2r-3}-q^{r-2}+q^{r-1}}=q^{2r-1}, \]
\[ N_{q^{2r-3}-q^{r-2}}=q^{2r}-q^{2r-1}. \]
In this case   the  weight enumerator of $\cC$ is
  $w(x):=\sum_i A_ix^i$ where
  \[ A_0=1,\qquad  A_{q^{2r-1}}=q^{2r}-1, \]
  and all of the remaining $A_i$'s are $0$ with the exception of
    \[ A_{q^{2r-1}-q^{2r-3}}=(q^2-1)q^{2r-1}, \quad
     A_{q^{2r-1}-q^{r-1}}=q^{2r-1}(q-1)(q^2-1). \]

\vskip.2cm\noindent
\begin{minipage}[t]{\textwidth}
Authors' addresses:
\vskip.2cm\noindent\nobreak
\centerline{
\begin{minipage}[t]{8cm}
Angela Aguglia \\
Department of Mechanics, Mathematics \\ and Management \\
Politecnico di Bari \\
Via Orabona 4, I-70126 Bari (Italy) \\
angela.aguglia@poliba.it\\
\end{minipage}\hfill
\begin{minipage}[t]{7cm}
Luca Giuzzi\\
D.I.C.A.T.A.M. \\ Section of Mathematics \\
Universit\`a di Brescia\\
Via Branze 43, I-25123, Brescia  (Italy) \\
luca.giuzzi@unibs.it
\end{minipage}}
\end{minipage}


\begin{thebibliography}{999}
\bibitem{AA} A. Aguglia,{\em Quasi-Hermitian varieties in $\PG(r,q^2)$, $q$  even}, Contrib. Discrete Math. {\bf 8} (2013), no. 1, 31--37.
\bibitem{ACK} A. Aguglia, A. Cossidente, G. Korchm\'{a}ros,  {\em On quasi-Hermitian varieties}, J. Combin. Des. {\bf 20} (2012), no. 10, 433--447.
\bibitem{AG1} A. Aguglia, L. Giuzzi, \emph{Intersections of the
    Hermitian surface with irreducible quadrics in $\PG(3,q^2)$, $q$ odd},
  Finite Fields Appl. {\bfseries 30} 1--13 (2014).
\bibitem{AG2} A. Aguglia, L. Giuzzi, \emph{Intersections of the
    Hermitian surface with irreducible quadrics in even characteristic},
  \emph{preprint} (arXiv:1407.8498)
\bibitem{AGK} A. Aguglia, L. Giuzzi, G. Korchm\'aros,
  \emph{Construction of unitals in Desarguesian planes},
  Discrete Math. {\bf 310} (2010), 3162--3167.
\bibitem{AK} A. Aguglia, G. Korchm\'aros, \emph{Multiple blocking sets and multisets in Desarguesian planes}, Des. Codes Cryptogr. {\bfseries 56}, 177--181 (2010).
\bibitem{BDG} D. Bartoli, A. A. Davydov, M. Giulietti, S. Marcugini, F. Pambianco, {\em Multiple coverings of the farthest-off points with small density from projective geometry}, Adv. Math. Commun. {\bf 9} (2015), no. 1, 63--85.
 \bibitem{E} S. Barwick, G. Ebert, \emph{Unitals in Projective Planes},
Springer Monographs in Mathematics (2008).
\bibitem{CK} A. Cossidente, O. H. King, {\em Some two-character sets}, Des. Codes Cryptogr. {\bf 56} (2010), no. 2-3, 105--113.
\bibitem{CP} A. Cossidente, T. Penttila, {\em Two-character sets arising from gluings of orbits}, Graphs Combin.  {\bf 29} (2013), no. 3, 399--406.
\bibitem{CaG} A. R. Calderbank,, Goethals, J. M. {\em Three-weight codes and association schemes}, Philips J. Res. {\bf 39} (1984), no. 4-5, 143--152.
\bibitem{CaK} R. Calderbank, W. M.  Kantor, {\em The geometry of two-weight codes}, Bull. London Math. Soc. {\bf 18} (1986), no. 2, 97--122.
\bibitem{Cha} I. M., Chakravarti, {\em Geometric construction of some families of two-class and three-class association schemes and codes from nondegenerate and degenerate Hermitian varieties}, Graph theory and combinatorics (Marseille-Luminy, 1990). Discrete Math. {\bf 111} (1993), no. 1-3, 95--103.
\bibitem{DoSi} S. Dodunekov, J. Simonis, {\em Codes and projective multisets}, Electron. J. Combin. {\bf 5} (1998), Research Paper 37.
\bibitem{Di0} C. Ding, C. Li, N. Li, Z. Zhou, {\em Three-weight cyclic
   codes and their weight distributions}, Discrete Math. {\bf 339} (2016),
 415--427.
\bibitem{Di1} K. Ding, C. Ding, {\em Binary Linear Codes With Three Weights},
IEEE Comm. Letters {\bf 18} (2014), 1879-1882.
\bibitem{Di}  K. Ding, C. Ding, {\em A class of two-weight and three-weights codes and their applications in secret sharing} arXiv:1503.065123.

\bibitem{HK} R. Hill, E. Kolev, {\em A survey of recent results on optimal linear codes}, Combinatorial designs and their applications (Milton Keynes, 1997), 127--152, Chapman and Hall/CRC Res. Notes Math., 403, Chapman and Hall/CRC, Boca Raton, FL, 1999.

\bibitem{HT} J.W.P. Hirschfeld, J. A.  Thas, {\em General Galois Geometries}, Second Edition,
Springer-Verlag (2015).
\bibitem{S} B. Segre,
  \emph{Forme e geometrie Hermitiane, con particolare riguardo al caso
    finito}, Ann. Mat. Pura Appl. (4) {\bfseries 70}, 1--201 (1965).
\bibitem{TVN} M.A. Tsfasman, S.G. Vl{\u{a}}du{\c{t}}, D.Yu. Nogin,
  \emph{Algebraic geometric codes: basic notions},
  Mathematical Surveys and Monographs {\bf 139},
  American Mathematical Society (2007).
\bibitem{DCS} H.N. Ward, \emph{Divisible codes -- a survey},
  Serdica Math. J. {\bfseries 27}, 263--278 (2001).
\end{thebibliography}
\end{document}